\documentclass[pdflatex, sn-basic]{sn-jnl}% Basic Springer Nature Reference Style/Chemistry Reference Style
\usepackage{graphicx}%
\usepackage{multirow}%
\usepackage{amsmath,amssymb,amsfonts}%
\usepackage{amsthm}%
\usepackage{mathrsfs}%
\usepackage[title]{appendix}%
\usepackage{xcolor}%
\usepackage{textcomp}%
\usepackage{manyfoot}%
\usepackage{booktabs}%
\usepackage{algorithm}%
\usepackage{algorithmicx}%
\usepackage{algpseudocode}%
\usepackage{listings}%
\usepackage{tikz}
\newcommand{\argmin}{\mathop{\rm argmin}\limits}
\newcommand{\prox}{{\rm prox}}

\newcommand{\R}{{\mathbb R}}

\newcommand{\calA}{{\mathcal A}}
\newcommand{\calB}{{\mathcal B}}
\newcommand{\calC}{{\mathcal C}}

\newtheorem{theorem}{Theorem}
\newtheorem{proposition}{Proposition}

\newtheorem{lemma}{Lemma}
\newtheorem{definition}{Definition}

\renewcommand{\proofname}{Proof}

\begin{document}

\title[Newton-type Methods with the Proximal Gradient Step for Sparse Estimation]{Newton-type Methods with the Proximal Gradient Step for Sparse Estimation}


\author*[1]{\fnm{Ryosuke} \sur{Shimmura}}\email{shimmura@sigmath.es.osaka-u.ac.jp}

\author[1]{\fnm{Joe} \sur{Suzuki}}\email{j-suzuki.jyou.es@osaka-u.ac.jp}

\affil[1]{\orgdiv{Graduate School of Engineering Science}, \orgname{Osaka University}, \orgaddress{\street{Toyonaka}, \postcode{560-8531},  \country{Japan}}}

\abstract{
In this paper, we propose new methods to efficiently solve convex optimization problems encountered in sparse estimation, which include a new quasi-Newton method that avoids computing the Hessian matrix and improves efficiency, and we prove its fast convergence. We also prove the local convergence of the Newton method under weaker assumptions. Our proposed methods offer a more efficient and effective approach, particularly for $L_1$ regularization and group regularization problems, as they involve variable selection with each update.
Through numerical experiments, we demonstrate the efficiency of our methods in solving problems encountered in sparse estimation. Our contributions include theoretical guarantees and practical applications for various problems.
}


\maketitle

\section{Introduction}
Statistics and machine learning are widely used in various fields, including science, engineering, and business, to analyze large datasets and extract meaningful insights. One common problem in these fields is to identify the most significant variables related to a predictor, which allows for accurate predictions and a better understanding of the underlying relationships. This problem is often encountered in sparse estimation, where a large number of variables must be considered, and only a small subset of them is expected to be relevant.

In this paper, we propose a new method to efficiently solve convex optimization problems encountered in statistics and machine learning, particularly those involving sparse estimation. We consider optimization problems of the form:
\begin{align}
    \label{eq:optp}
    \min_{x\in \R^n} f(x) + g(x)
\end{align}
for convex $f:\R^n \rightarrow \R,\;g: \R^n \rightarrow (-\infty, \infty]$, where $f$ is a loss function that is twice differentiable and $\mu$-strongly convex ($\mu>0$) and $g$ is a regularization term that is closed convex. We define strong convexity in Section \ref{subsec:convexity}.
Most sparse estimation problems can be formulated as (\ref{eq:optp}). For example, for $\lambda>0$, $f(x)=\|A x-b\|_{2}^{2}$ and $g(x)=\lambda\|x\|_{1}$ in lasso \citep{lasso} and $g(x)=\lambda\|x\|_{2}$ in group lasso \citep{grouplasso}, where $\|\cdot\|_{2}$ and $\|\cdot\|_{1}$ are the $L_{2}$-norm and $L_{1}$-norm, respectively.
To solve this problem efficiently, we propose methods to find a fixed point of the proximal gradient method, which can be used more broadly than in sparse estimation.

The proximal gradient and proximal Newton methods are commonly used to solve similar optimization problems, but they have limitations. The proximal gradient method can perform each update quickly, but it converges slowly and requires many updates. On the other hand, the proximal Newton method converges rapidly, but the computational cost of each update becomes high. Moreover, there are some issues with the efficiency of the proximal Newton method, mainly when applied to group sparsity problems.

In recent research, a method for finding fixed points of the proximal gradient method has been discussed in \citep{xiao2018regularized}. The semismooth Newton method can be used to solve this problem, and it is algorithmically equivalent to the approach proposed in \citep{patrinos2013proximal, patrinos2014forward, stella2017forward}. Additionally, stochastic methods have been suggested as an alternative \citep{milzarek2019stochastic, yang2021stochastic}. However, both methods require the Lipschitz constant for the first derivative of the objective function ($\nabla f$), and there have been no reports on their convergence when this constant is unknown or absent. In this study, we prove the convergence of the semismooth Newton method when the Lipschitz condition of $\nabla f$ is eliminated, and we extend the theory. Recently, a similar method using the semismooth Newton method has been proposed to find a fixed point of ADMM, which can efficiently obtain high-precision solutions \citep{li2017semi, ali2017semismooth}.

To overcome these limitations, we propose new methods that finds the fixed point of the proximal gradient method efficiently, even when the Lipschitz constant is unknown. We also extend the theory to prove the convergence of the semismooth Newton method under such conditions. Additionally, we introduce a new quasi-Newton method that approximates only the second derivative of the loss function to avoid computing the Hessian matrix and improve efficiency.

The main contributions of this study are: (1) proving the local convergence of the semismooth Newton method under weaker assumptions, (2) proposing a new quasi-Newton method that avoids computing the Hessian matrix and establishing its superlinear convergence, and (3) demonstrating the efficiency of the proposed methods in solving convex optimization problems encountered in statistics and machine learning through numerical experiments.

Overall, our proposed methods offer a more efficient and effective way to solve convex optimization problems encountered in sparse estimation. 
Especially in sparse estimation techniques such as $L_1$ regularization and group regularization, our proposed method can efficiently find solutions by performing variable selection using the proximal gradient method with each update.

The remainder of the paper is organized as follows. In Section \ref{sec:background}, we provide background knowledge for understanding this paper. Section \ref{sec:newton} presents the semismooth Newton method and proves its local convergence. Section \ref{sec:quasinewton} presents the new quasi-Newton method and proves its local convergence. In Section
\ref{sec:experiment}, we empirically evaluate the performance of the proposed methods.
Finally, Section \ref{sec:conclusion} summarizes the results of this paper and describes future work.

\section{Background}
\label{sec:background}
\subsection{Convex Function and Its Subdifferential}
In this section, we provide background information for understanding the results in the later sections.
\label{subsec:convexity}
We say that a function $f:\R^n \rightarrow \R$ is {\it convex} if
\begin{align}
    \label{totsusei}
    f((1-\lambda)x+\lambda y)\leq(1-\lambda)f(x)+\lambda f(y)
\end{align}
for any $x,y\in \R^n$ and $0\leq \lambda \leq 1$.
In particular, we say that the convex function $f$ is {\it closed} if
$\{x\in {\mathbb R}^n\vert f(x)\leq \alpha \}$ is a closed set for each $\alpha\in {\mathbb R}$.
Moreover, we say that $f$ is {\it $\mu$-strongly convex} if
\begin{align}
    f((1-\lambda)x+\lambda y)\leq(1-\lambda)f(x)+\lambda f(y) - \frac{\mu}{2}\lambda (1- \lambda) \|x-y\|_2^2
\end{align}
for any $x,y \in \R^n$ and $0 \leq \lambda \leq 1$. When $f$ is twice differentiable, $f$ is {\it $\mu$-strongly convex} if and only if $\nabla^2 f(x) - \mu I$ is positive semidefinite for any $x\in \R^n$ \citep{bauschke2011convex}.
%closed if for each {\displaystyle \alpha \in \mathbb {R} }\alpha \in {\mathbb  {R}}, the sublevel set

For a convex function $f:\R^n\rightarrow \R$, we define the {\it subdifferential} of $f$ at $x_0\in \R^n$ by the set of $z\in \R^n$ such that
\begin{align}
	\label{retsubi}
	f(x)\geq f(x_0)+\langle z , x-x_0 \rangle
\end{align}
for any $x\in \R^n$ and denote it as $\partial f(x_0)$.
For example, the subdifferential of $f(x)=|x|,x\in \R$ at $x=0$ is the set of $z$ such that $|x|\geq zx,x\in \R$,
and we write $\partial f(0)=\{z\in \R \mid |z|\leq 1\}$.

\subsection{Proximal Gradient Method}
\label{subsec:PG}
The proximal gradient method finds the minimum solution of an objective function expressed as the sum of convex functions $f,g$ that
are differentiable and not necessarily differentiable, respectively.
We define the functions
\begin{align}
	\label{kinsetuq}
	Q_\eta(x,y)&:=f(y)+\langle x-y,\nabla f(y)\rangle +\frac{1}{2\eta}\|x-y\|_2^2+g(x)\\
	\label{kinsetup}
	p_\eta(y)&:=\argmin_x \;\; Q_\eta (x,y)
\end{align}
for $\eta>0$ and generate the sequence $\{x^{(k)}\}$ via
\begin{equation}
        \label{eq:proxup}
	x^{(k+1)}:= \prox_{\eta g}(y-\eta \nabla f(x^{(k)}))
\end{equation}
from the initial value $x^{(0)}$ until convergence to obtain the solution.
If we define the {\it proximal map} with respect to $h:\R^n\rightarrow (-\infty, \infty]$ by
\begin{align}
	\label{kinsetusyazou}
	\prox_{h}(x)=\argmin_x \left\{h(x)+\frac{1}{2}\|y-x\|_2^2 \right\},
\end{align}
then (\ref{kinsetup}) can be expressed as
\begin{align}
	p_{\eta}(y)&=\argmin_x Q_\eta (x,y)\nonumber \\
	&=\argmin_x \left\{\langle x-y, \nabla f(y) \rangle +\frac{1}{2\eta}\|x-y\|_2^2+g(x)\right\} \nonumber \\
	&=\argmin_x \left\{g(x)+\frac{1}{2\eta}\|x-y-\eta \nabla f(y)\|_2^2\right\} \nonumber \\
	\label{kinprox}
	&=\prox_{\eta g}(y-\eta \nabla f(y)) . 
\end{align}
In each iteration, the proximal gradient is used to search for $x$ that minimizes
the sum of the quadratic approximation of $f(x)$ around $x^{(k)}$ and $g(x)$.
The iterative shrinkage-thresholding algorithm (ISTA) procedure obtains $O(k^{-1})$ accuracy for $k$ updates \citep{Beck2009}.

Even if we replace the update (\ref{eq:proxup}) with the ISTA,
they do not necessarily converge to $x$, which minimizes the objective function unless we choose an appropriate parameter $\eta$.
In the following, we assume that $\nabla f$ is Lipschitz continuous, which means that there exists $L_f>0$ such that for arbitrary $x,y\in \R^n$,
\begin{align}
	\label{lipschitz}
	\| \nabla f(x)-\nabla f(y) \|_2\leq L_f \| x -y \|_2 \ .
\end{align}
It is known that the ISTA converges to $x$ that minimizes the objective function
if we choose $\eta>0$ as $0<\eta\leq L_f^{-1}$ \citep{Beck2009}.

\subsection{Proximal Newton Method}

The proximal Newton method finds the minimum solution of (\ref{eq:optp}). In each iteration $k$, the proximal Newton method approximates $f(x)$ as
\begin{align}
    f(x)\approx f(x^{(k)})+ \nabla f(x^{(k)})^T x + \frac{1}{2}x^T \nabla^2 f(x^{(k)})x 
\end{align}
and generates the sequence $\{x^{(k)}\}$ via
\begin{align}
    \label{eq:proxnew}
    \tilde{x}^{(k)} &= \argmin_{x\in \R^n} f(x^{(k)})+ \nabla f(x^{(k)})^T x + \frac{1}{2}x^T \nabla^2 f(x^{(k)})x + g(x) \\
    x^{(k+1)} &= x^{(k)} + \eta_k (\tilde{x}^{(k)} - x^{(k)})  ,\;\;\;  0 < \eta_k \leq 1
\end{align}
from the initial value $x^{(1)}$ until convergence to obtain the solution. In this paper, we consider only the case where $\eta_k = 1$. If $g(x)= \lambda \|x\|_1\; (\lambda \geq  0)$, subproblem (\ref{eq:proxnew}) is equivalent to a lasso problem \citep{lasso}, and $x$ can be updated efficiently by the coordinate descent method.

However, if $g(x) = \lambda \|x\|_2$, the subproblem (\ref{eq:proxnew}) is equivalent to a group lasso problem \citep{grouplasso}, which is solved using the proximal gradient method; hence, the high convergence speed of the proximal Newton method cannot be realized.

\subsection{Optimality Conditions}
\begin{proposition}[\rm{\cite{xiao2018regularized}, Lemma 2.1}]
    \label{prop:opt}
    Suppose that $f:\R^n \rightarrow \R$ is a differentiable convex function and $g:\R^n \rightarrow (-\infty,\infty]$ is a closed convex function. Then, the following are equivalent for all $\nu > 0$:
    \begin{align}
        x^* &\in \argmin_{x\in \R^n} f(x)+g(x) \\
        0 &\in \nabla f(x^*) + \partial g(x^*) \\
        x^* &= \prox_{\nu g} (x^* - \nu \nabla f(x^*))
    \end{align}
\end{proposition}
We define the function $F_\nu : \R^n \rightarrow \R^n$ for $\nu > 0$ as
\begin{align}
    \label{eq:Fnu}
    F_\nu (x):= x- \prox_{\nu g} (x - \nu \nabla f(x)).
\end{align}
From Proposition \ref{prop:opt}, $x$ such that $F_\nu (x) = 0$ minimizes (\ref{eq:optp}). Therefore, by solving the nonlinear equation $F_\nu (x) = 0$, we can find $x$ that minimizes (\ref{eq:optp}). In this paper, we consider Newton and quasi-Newton methods that solve $F_\nu (x) = 0$ for all $\nu > 0$. However, previous research \citep{patrinos2013proximal,patrinos2014forward,stella2017forward,xiao2018regularized} considered Newton methods for $\nu \leq 2L_f^{-1}$.

In addition, considering the updating equation of the proximal gradient method
\begin{align}
    x^{(k+1)}= \prox_{\nu g} (x^{(k)}-\nu \nabla f(x^{(k)})),
\end{align}
$F_\nu$ can be interpreted as the difference $F_\nu (x^{(k)}) =x^{(k)} -x^{(k+1)}$ in the proximal gradient method.

\subsection{Linear Newton Approximations}

\label{subsec:LNA}
\begin{definition}
If $\mathcal{A}(x)$ is a subset of $\R^{n \times n}$ for each $x\in \R^n$, then $\mathcal{A}$ is called a {\it set-valued function}, and we write $\mathcal{A}:\R^n \rightrightarrows \R^{n\times n}$. A set-valued function $\mathcal{A} \R^n \rightrightarrows \R^{n\times n}$ is {\it upper-semicontinuous} at $x\in \R^n$ if for any $\epsilon > 0 $, there exists $\delta > 0 $ such that for all $y\in \R^n$,
\begin{align}
    \|x-y\|_2<\delta \; \Rightarrow \; \mathcal{A}(y) \subset \mathcal{A}(x) + {\mathbb B}(O,\epsilon),
\end{align}
where $O$ is a matrix with all elements zero, ${\mathbb B}(B, \delta) := \{A\in \R^{n\times n} \mid \|A-B\| < \delta \}$, and $\|\cdot \|$ denotes the operator norm.
\end{definition}

We will apply the Newton method using the derivative with respect to $F_\nu$ defined in (\ref{eq:Fnu}), but in general, $\prox_{\nu g}$ is not differentiable.
However, $\prox_{\nu g}$ is Lipschitz continuous with parameter 1, which means that $\forall x\in \R^n, \forall y \in \R^n , \|\prox_{\nu g}(x)-\prox_{\nu g}(y)\|_2 \leq \|x-y\|_2$. Thus, we define a B-subdifferential, which generalizes the derivative for Lipschitz continuous functions, as follows.

\begin{definition}
    Let $F:\R^n \rightarrow \R^n$ be Lipschitz continuous. The B-subdifferential of $F$ at $x\in \R^n$ is
    \begin{align}
        \label{eq:B-dif}
		\partial_B F(x) = \left\{ V\in \R^{n\times n} \mid \exists \{x^{(k)}\} \subset \mathcal{D}_F , such\; that \;\; x^{(k)}\rightarrow x, \;\; \nabla F(x^{(k)}) \rightarrow V \right\},
    \end{align}
    where $\mathcal{D}_F$ is the subset of $\R^n$ for which $F$ is differentiable, i.e., the B-subdifferential is the set of $V$ such that there exists a sequence $\{x^{(k)}\}$ that satisfies the following three conditions: 1. $F$ is differentiable for all $x^{(k)}$, 2. $x^{(k)} \rightarrow x$, and 3. $\nabla F(x^{(k)}) \rightarrow V$.
\end{definition}    

If $F:\R^n \rightarrow \R^n$ is Lipschitz continuous, then $\partial_B F(x)$ is a nonempty and compact subset of $\R^{n \times n}$, and the set-valued function $\partial_B F$ is upper-semicontinuous at every $x\in \R^n$ \citep[proposition 2.2]{ulbrich2011semismooth}. If $F$ is differentiable at $x$, then $\partial_B F(x) = \{\nabla F(x) \}$. In particular, if $f:\R^n \rightarrow \R$ is twice differentiable at $x$, then $\partial_B\left(\nabla f(x) \right) = \{\nabla^2 f(x)\}$. In this paper, we approximate $F_\nu$ defined in (\ref{eq:Fnu}) using $\partial_B \prox_{\nu g}$, which is the B-subdifferential of $\prox_{\nu g}$. Thus, it is important to approximate $F_\nu$, for which we define the following linear Newton approximation.

\begin{definition}[\rm{\cite{facchinei2003finite}, Definition 7.5.13}]
    Let $F:\R^n \rightarrow \R^n$ be continuous. We say that a set-valued function $\calA:\R^n \rightrightarrows \R^{n \times n}$ is a linear Newton approximation (LNA) of $F$ at $x\in \R^n$ if $\calA$ has compact images and is upper-semicontinuous at $x$ and
    \begin{align}
        \|F(x)-F(y) - A(x-y)\|_2 = o(\|x-y\|_2) \;\; as \; \;y \rightarrow x.
    \end{align}
    for $y \in \R^n$ and any $A\in \calA(y)$. If instead
    \begin{align}
        \|F(x)-F(y) - A(x-y)\|_2 = O(\|x-y\|_2^2) \;\; as \; \; y \rightarrow x
    \end{align}
    for $y \in \R^n$ and any $A\in \calA(y)$, then we say that $\calA:\R^n \rightrightarrows \R^{n \times n}$ is a strong linear Newton approximation (strong LNA) of $F$ at $x\in \R^n$.
\end{definition}

If $F$ has an LNA, then there exists a matrix $A$ that can approximate $F(y)-F(x)$. For example, if $F(x) = x$ and $\calA(x) = \{I\}$ for any $x$, then
\[
    \|x-y-I(x-y)\|_2 = 0
\]
for any $x,y\in \R^n$ and $\calA$ is a strong LNA of $F(x)=x$ for every $x\in \R^n$. In general, if $f:\R^n \rightarrow \R$ is twice differentiable at $x\in \R^n$, then by using $\nabla^2 f: \R^n \rightarrow \R^{n \times n}$, which is the Hessian of $f$, and setting $\calB (x) = \{\nabla^2 f(x)\}$ for any $x\in \R^n$, we find that $\calB$ is an LNA of $\nabla f$ for every $x$ and
\[
    \|\nabla f(x)-\nabla f(y)-\nabla^2 f(y)(x-y)\|_2=o(\|x-y\|_2) \;\; as \;\; y\rightarrow x
\]
holds. In particular, if $\nabla^2 f$ is Lipschitz continuous, i.e., there exists $L_f>0$ such that $\| \nabla^2 f(x) -\nabla^2 f(y) \| \leq L_f \|x-y\|_2$ for any $x,y \in \R^n$, then $\calB$ is a strong LNA of $\nabla f$ for every $x$.
Here, since $I$ and $\nabla^2 f$ are continuous functions on $\R^n$, it is apparent that both $\calA$ and $\calB$ are upper semicontinuous.
However, if $F$ is not differentiable, we need to determine whether $\partial_B F$ is an LNA of $F$. In this paper, we construct an LNA of $\prox_{\nu g}$ using $\partial_B \prox_{\nu g}$.

An LNA has similar properties to ordinary derivatives, and the linearity and chain rule can be expressed as follows.

\begin{lemma}[\rm{\cite{facchinei2003finite}, Corollaly 7.5.18}]
    \label{lem:linear}
    Suppose that set-valued functions $\calA:\R^n \rightrightarrows \R^{n \times n},\calB:\R^n \rightrightarrows \R^{n \times n}$ are (strong) LNAs of $F:\R^n \rightarrow \R^n$ and $G:\R^n \rightarrow \R^n$, respectively, at $x\in \R^n$. Then,
    \[
        (\calA + \calB)(y):= \{A+B\mid A\in \calA(y),\; B \in \calB(y) \}
    \]
    is a (strong) LNA of $F+G$ at $x$.
\end{lemma}

\begin{lemma}[\rm{\cite{facchinei2003finite}, Theorem 7.5.17}]
    \label{lem:chain}
    Suppose that the set-valued function $\calB:\R^n \rightrightarrows \R^{n \times n}$ is a (strong) LNA of $G:\R^n \rightarrow \R^n$ at $x\in \R^n$ and that $\calA:\R^n \rightrightarrows \R^{n \times n}$ is a (strong) LNA of $F:\R^n \rightarrow \R^n$ at $G(x)$. Then,
    \[
        (\calA \calB)(y) := \{AB \mid A\in \calA(G(y)), \; B\in \calB(y) \}
    \]
    is a (strong) LNA of $F\circ G$ at $x$, where $F \circ G$ is the composition of the mappings $F\circ G(x) = F(G(x))$.
\end{lemma}

From Lemmas \ref{lem:linear} and \ref{lem:chain}, as in ordinary differential calculus, when the function for which an LNA is to be obtained is expressed as a sum of multiple functions or their composite map, it is sufficient to consider an LNA of each function. For example, we suppose the set-valued functions $\calA,\calB: \R^n \rightrightarrows \R^{n\times n}$ are (strong) LNAs of $F,G:\R^n \rightarrow \R^n$ at $x\in \R^n$ and that $calC:\R^n \rightrightarrows \R^{n\times n}$ is a (strong) LNA of $H:\R^n \rightarrow \R^n$ at $F(x)+G(x)$.
Then, 
\[
\calC (\calA + \calB)(y) := \{C(A+B) \mid C\in \calC(F(y)+G(y)),\; A\in \calA (y),\; B \in \calB(y) \}
\]
is a (strong) LNA of $H\circ (F+G)$ at $x$.

\subsection{LNA of the proximal map}

\subsubsection{$L_1$ norm}

If $g(x) = \| x\|_1$, the $i$-th component of $\prox_{\nu g}$ is
\begin{align}
    \label{eq:proxL1}
    \prox_{\nu g} (x)_i = \left(1-\frac{\nu}{|x_i|}\right)_+ x_i,
\end{align}
where $(s)_+ = \max\{0,s\}$ for $s\in \R$. (\ref{eq:proxL1}) is differentiable at any $|x_i| \neq \nu$, and its derivative is $0$ for $|x_i| < \nu$ and 1 for $|x_i| > \nu$.
For the case of $|x_i| = \nu$, 
if $x_i^{(k)} \rightarrow x_i, |x_i^{(k)}|\downarrow \nu$ as $k\rightarrow \infty$, then $\nabla \prox_{\nu g} (x^{(k)})_{i,i} \rightarrow 1$.
In contrast, if $x_i^{(k)} \rightarrow x_i, |x_i^{(k)}|\uparrow \nu$ as $k\rightarrow \infty$, then $\nabla \prox_{\nu g} (x^{(k)})_{i} \rightarrow 0$.
Thus, $\partial_B \prox_{\nu g}(x)$ becomes the set of diagonal matrices for any $x\in \R^n$, and its $(i,i)$-th component is
\begin{align}
    \label{eq:LNAL1}
    \partial_B \prox_{\nu g}(x)_{i,i} = 
	\begin{cases}
		\{0\} & |x_i|<\nu \\
		\{1\} & |x_i|>\nu \\
		\{0,1\} & |x_i|=\nu
	\end{cases}.
\end{align}

\subsubsection{$L_2$ norm}

If $g(x) = \|x\|_2$, $\prox_{\nu g}$ is
\begin{align}
    \label{eq:proxL2}
    \prox_{\nu g} (x)= \left(1- \frac{\nu}{\|x\|_2} \right)_+ x .
\end{align}
(\ref{eq:proxL2}) is differentiable at any $\|x\|_2 \neq \nu$, and its derivative is $O$ for $\|x\|_2 < \nu$ and $\frac{\nu}{\|x\|_2}(\frac{xx^T}{\|x\|_2^2}-I) + I$ for $\|x\|_2 > \nu$.
For the case of $\|x\|_2 = \nu$,
if $x^{(k)} \rightarrow x, \|x^{(k)}\| \downarrow \nu$ as $k\rightarrow \infty$, then $\nabla \prox_{\nu g} (x^{(k)})  \rightarrow \frac{\nu}{\|x\|_2}(\frac{xx^T}{\|x\|_2^2}-I) + I$.
In contrast, if $x^{(k)} \rightarrow x, \|x^{(k)}\|_2\uparrow \nu$ as $k\rightarrow \infty$, then $\nabla \prox_{\nu g} (x^{(k)}) \rightarrow O$.
Thus, $\partial_B \prox_{\nu g}(x)$ is the set of symmetric matrices for any $x\in \R^n$, and 
\begin{align}
    \label{eq:LNAL2}
    \partial_B \prox_{\nu g}(x) = 
    \begin{cases}
	\{O\} & \|x\|_2<\nu \\
	\left\{\frac{\nu}{\|x\|_2}(\frac{xx^T}{\|x\|_2^2}-I) + I\right\} & \|x\|_2>\nu \\
	\left\{O,\frac{xx^T}{\|x\|_2^2}\right\} & \|x\|_2=\nu
    \end{cases}
    .
\end{align}

\begin{lemma}[\cite{zhang2020efficient}, Lemma 2.1]
    \label{lem:l1l2sLNA}
    (\ref{eq:LNAL1}) and (\ref{eq:LNAL2}) are strong LNAs of $\prox_{\nu \| \cdot \|_1}$ and $\prox_{\nu \| \cdot \|_2}$, respectively, for any $x\in \R^n$.
\end{lemma}

From Lemma \ref{lem:l1l2sLNA}, if $g$ is either the $L_1$-norm or the $L_2$-norm, then the B-subdifferential is a strong LNA of $\prox_{\nu g}$.
Moreover, if $\prox_{\nu g}(x)_i$ is 0, i.e. $x_i$ is inactive, then the corresponding component of $\partial_B \prox_{\nu g}$ is 0.

\section{Linear Newton Method}

\label{sec:newton}
Here, we consider the linear Newton method for solving $F_\nu(x) = 0$ for any $\nu > 0$. According to Proposition \ref{prop:opt}, $x$ such that $F_\nu (x) = 0$ minimizes (\ref{eq:optp}). First, we consider an LNA of $F_\nu$ to execute the linear Newton method. From Lemmas \ref{lem:linear} and \ref{lem:chain}, since the LNA is linear and satisfies the chain rule, we define the set-valued function $\partial F_\nu : \R^n \rightrightarrows \R^{n\times n}$ as
\begin{align}
    \label{eq:FnuLNA}
    \partial F_\nu(x) = \left\{ I- V(I-\nu \nabla^2 f(x)) \mid V \in \partial_B \prox_{\nu g} (x-\nu \nabla f(x))\right\},
\end{align}
which is an LNA of $F_\nu (x) = x-\prox_{\nu g}(x-\nu \nabla f(x))$. We show that (\ref{eq:FnuLNA}) is a (strong) LNA of $F_\nu$ as follows.
\begin{proposition}
    \label{prop:LNA}
    Let $x\in \R^n$. If $\partial_B \prox_{\nu g}$ is an LNA of $\prox_{\nu g}$ at $x-\nu \nabla f(x)$, then $\partial F_\nu$ is an LNA of $F_\nu$ at $x$. Moreover, if $\partial_B \prox_{\nu g}$ is a strong LNA of $\prox_{\nu g}$ at $x-\nu \nabla f(x)$ and $\nabla^2 f$ is Lipschitz continuous, then $\partial F_\nu$ is a strong LNA of $F_\nu$ at $x$.
\end{proposition}
\begin{proof}
    $\mathcal{A}:\R^n \ni y \mapsto \{I\}$ is an LNA of $F:\R^n \ni y \mapsto y \in \R^n$ at $x\in \R^n$. Since $\nabla f$ is differentiable, $\calB:\R^n \ni y \mapsto \{-\nu \nabla^2 f(y)\}$ is an LNA of $G:\R^n \ni y \mapsto -\nu \nabla f(y) \in \R^n$ at $x$. Thus, from Lemma \ref{lem:linear}, $\calA + \calB = \{I - \nu \nabla^2 f\}$ is an LNA of
    \[
        F+G : \R^n \ni y \mapsto y-\nu \nabla f(y) \in \R^n
    \]
    at $x$. By assumption, $\calC :\R^n \ni y \mapsto \partial_B \prox_{\nu g}(y)$ is an LNA of $H: \R^n \ni y \mapsto \prox_{\nu g}(y)$ at $x-\nu \nabla f(x)$. From Lemma \ref{lem:chain}, if we define the set-valued function $\partial P_\nu : \R^n \rightrightarrows \R^{n\times n}$ as 
    \[
        \partial P_\nu (x):= \calC(\calA + \calB) = \left\{ V(I-\nu \nabla^2f)\right\},
    \]
    then $\partial P_\nu$ is an LNA of $H\circ(F+G) : \R^n \ni y \mapsto \prox_{\nu g}(y-\nu \nabla f(y))\in \R^n$ at $x$. Furthermore, applying Lemma \ref{lem:linear}, we find that $\partial F_\nu$ is an LNA of $F_\nu$, which proves the first claim. If $\nabla^2 f$ is Lipschitz continuous, then $\calB$ is a strong LNA of $G$ at $x$, so we conclude that $\partial F_\nu$ is a strong LNA of $F_\nu$ at $x$.
\end{proof}

From Proposition \ref{prop:LNA}, if $\partial_B \prox_{\nu g}$ is a (strong) LNA of $\prox_{\nu g}$, since $g$ is the $L_1$ regularization or a group regularization, then $\partial F_\nu$ is an LNA of $F_\nu$ and can approximate $F_\nu$. Thus, we conclude that $\partial_B \prox_{\nu g}$ is an LNA of $\prox_{\nu g}$, which is important for the discussion below. Next, we consider the linear Newton method based on $\partial F_\nu$.

\subsection{Procedure}
\label{subsec:alLNM}
From Proposition \ref{prop:LNA}, which states that $\partial F_\nu$ is an LNA of $F_\nu$, we can approximate $F_\nu$ using $\partial F_\nu$. Specifically, we approximate
\begin{align}
    \label{eq:approx}
    F_\nu(x)\approx F_\nu^{(k)}(x):= F_\nu(x^{(k)}) + U^{(k)} (x - x^{(k)})  ,\;\;\; U^{(k)} \in \partial F_\nu (x^{(k)}) 
\end{align}
for $k=1,2,\ldots$ with the initial value $x^{(1)}\in \R^n$, and we update $x^{(k+1)}$ such that $F_\nu^{(k)} (x^{(k+1)})= 0$. We show the procedure of the linear Newton method in Algorithm \ref{alg:L-newton}.
\begin{algorithm}[h]
	\caption{(Linear Newton Method) input: $x^{(1)}$, output: $x^{\infty}$}
	\label{alg:L-newton}
        $\nu>0$ is initialized, and steps \ref{enum:LNMd} and \ref{enum:LNMx} are repeated for $k = 1,2,\ldots$ until convergence.
	\begin{enumerate}
		
		\item (Obtaining $d$)
		\label{enum:LNMd}
		$I-V^{(k)}(I-\nu \nabla^2 f(x^{(k)}))\in \partial F_\nu (x^{(k)})$ for $V^{(k)} \in \partial_B \prox_{\nu g}(x^{(k)}-\nu \nabla f(x^{(k)}))$ is chosen, and the selected search direction $d^{(k)}\in \R^n$ is
		\begin{align}
			\label{eq:newtonup}
			d^{(k)} := -\left(I-V^{(k)}\left(I-\nu \nabla^2 f(x^{(k)})\right) \right)^{-1}F_\nu(x^{(k)}).
		\end{align}
		\item (Updating $x$)
        \label{enum:LNMx}
		\begin{align}
			x^{(k+1)} := x^{(k)} + d^{(k)}.
		\end{align}
        \item When $x^{(k)}$ converges,
        \begin{align*}
            x^{\infty} = x^{(k+1)}.
        \end{align*}
	\end{enumerate}
\end{algorithm}

In this paper, we prove the following proposition, namely, Proposition \ref{prop:nonsingular}, to guarantee that the inverse matrix of $I - V^{(k)}\left(I-\nu \nabla^2 f(x^{(k)})\right)$ always exists and that the update of (\ref{eq:newtonup}) is always possible. The proof is presented in detail in the appendix.
\begin{proposition}
    \label{prop:nonsingular}
    Suppose that $f:\R^n \rightarrow \R$ and $g:\R^n \rightarrow (-\infty, \infty]$ are $\mu$-strongly convex and closed convex, respectively. Then, $I-V\left( I - \nu \nabla^2 f(x) \right)$ is a nonsingular matrix for any $x\in \R^n, \nu>0$ and $V\in \partial_B \prox_{\nu g}(x-\nu \nabla f(x))$, and all eigenvalues are positive real numbers.
\end{proposition}

It was known in \citep{xiao2018regularized} that if $\nu < L_f^{-1}$, then all eigenvalues of $I-V\left(I - \nu \nabla^2 f(x)\right)$ are nonnegative real numbers, but it was not known whether they are nonsingular or singular. Proposition \ref{prop:nonsingular} in this paper shows that for general $\nu > 0$, $I-V\left(I - \nu \nabla^2 f(x)\right)$ is always nonsingular for any $x\in \R^n$ if $f$ is $\mu$-strongly convex. Thus, since any element of $\partial F_\nu(x^{(k)})$ is nonsingular for each iteration $k$, there exists $d^{(k)}$ such that (\ref{eq:newtonup}) is satisfied and the update of Algorithm \ref{alg:L-newton} is always possible.

\subsection{$L_1$ Regularization}
Based on (\ref{eq:LNAL1}), we define the diagonal matrix $V^{(k)} \in \partial_B \prox_{\nu g} (x^{(k)}-\nu \nabla f(x^{(k)}))$ as
\begin{align}
    V_{i,i}^{(k)} = 
    \begin{cases}
		0,  & |x^{(k)}_i- \nu \nabla f(x^{(k)})|\leq\nu \lambda \\
		1,  & |x^{(k)}_i-\nu \nabla f(x^{(k)})|>\nu \lambda\\
    \end{cases}
\end{align}
and obtain $I-V^{(k)} \left( I - \nu \nabla^2 f(x^{(k)} \right)\in \partial F_\nu (x^{(k)})$. If we define the index sets $\mathcal{I}^{(k)}, \mathcal{O}^{(k)}$ as
\begin{align*}
	\mathcal{I}^{(k)} &= \{i | V_{ii}^{(k)} = 1\}\\
	\mathcal{O}^{(k)} &= \{i | V_{ii}^{(k)} = 0\},
\end{align*}
then we can express the matrix as
\begin{align}
    I-V^{(k)}(I-\nu \nabla^2 f(x^{(k)})) = 
	\begin{pmatrix}
		\nu \nabla^2 f(x^{(k)})_{\mathcal{I}^{(k)},\mathcal{I}^{(k)}} && \nu \nabla^2 f(x^{(k)})_{\mathcal{I}^{(k)},\mathcal{O}^{(k)}}  \\
		O && I
	\end{pmatrix},
\end{align}
where $\nabla^2 f(x^{(k)})_{\mathcal{I}^{(k)},\mathcal{I}^{(k)}}$ and $\nabla^2 f(x^{(k)})_{\mathcal{I}^{(k)},\mathcal{O}^{(k)}}$ are the elements of the matrix in $(\mathcal{I}^{(k)}, \mathcal{I}^{(k)})$ and $(\mathcal{I}^{(k)}, \mathcal{O}^{(k)})$, respectively. Thus, we can efficiently update by eliminating the calculation for the components $i$ such that $V_{i, i}^{(k)}=0$, i.e., $\prox_{\nu g} (x^{(k)}- \nu \nabla f(x^{(k)}))=0$.

\subsection{Convergence}
We consider the convergence properties of Algorithm \ref{alg:L-newton}.

\begin{theorem}
    \label{thm:convergencenewton}
    Suppose $\partial F_\nu$ is an LNA of $F_\nu$ at the optimal solution $x^*$ and all elements of $\partial F_\nu (x)$ are nonsingular for any $x\in \R^n $. Then, the sequence $\{x^{(k)}\}_{k=1}^\infty$ generated by Algorithm \ref{alg:L-newton} converges locally superlinearly to $x^*$ such that $F_\nu (x^*) = 0$. Moreover, if $\partial F_\nu$ is a strong LNA of $F_\nu$ at the optimal solution $x^*$, then the sequence $\{x^{(k)}\}_{i=1}^\infty$ generated by Algorithm \ref{alg:L-newton} converges locally quadratically to $x^*$.
\end{theorem}
\begin{proof}
    By assumption, since $\partial F_\nu$ is a (strong) LNA of $F_\nu$ at the optimal solution $x^*$ and $A$ is a nonsingular matrix for any $x\in \R^n $ and $A\in \partial F_\nu(x)$, the proof of this theorem follows from Theorem 2.11 in reference \citep{hintermuller2010semismooth} and Theorem 7.5.15 in reference \citep{facchinei2003finite}.
\end{proof}
From Proposition \ref{prop:nonsingular}, if $f$ is $\nu$-strongly convex, all elements of $\partial F_\nu (x)$ are nonsingular for any $x\in \R^n$. Thus, in the cases of $L_1$ regularization, group regularization, etc., local quadratic convergence is achieved when updating with Algorithm \ref{alg:L-newton} due to Theorem \ref{thm:convergencenewton}. Here, Theorem \ref{thm:convergencenewton} suggests that the parameter $\nu$ of $F_\nu$ is arbitrary as long as $\nu > 0$. Therefore, Theorem \ref{thm:convergencenewton} shows the convergence of Algorithm \ref{alg:L-newton} in the general case without requiring $\nu \leq 2 L_f^{-1}$ as in \citep{xiao2018regularized}, thereby extending previous results.

\section{Hybrid Linear Quasi-Newton Method}
\label{sec:quasinewton}
Executing Algorithm \ref{alg:L-newton} is often hard because the calculation of $\nabla^2 f$ is time-consuming and may not be possible. Thus, we consider approximating $\partial F_\nu (x^{(k)})$ by $x^{(k)}$ in each iteration $k$. In this paper, we consider approximating $\nabla^2 f(x^{(k)})$ to make the computation of $\partial F_\nu (x^{(k)})$ feasible while having the linear Newton method's advantage that the calculation can be omitted when $x$ equals 0. Specifically, we define a new set-valued function $\hat{\partial}^{(k)}F_\nu:\R^n \rightrightarrows \R^{n \times n}$ as
\[
\hat{\partial}^{(k)}F_\nu (x^{(k)}) =  \left\{ I- V(I-\nu B^{(k)}) \mid V \in \partial_B \prox_{\nu g} (x^{(k)}-\nu \nabla f(x^{(k)}))\right\}
\]
using an approximation matrix $B^{(k)}$ at each iteration and approximate $F_\nu$ as in (\ref{eq:approx}).

\subsection{Procedure}
Since $\partial F_\nu$ is an LNA of $F_\nu$, if $B^{(k)}$ can successfully approximate $\nabla^2 f(x^{(k)})$ in each iteration $k$, then $\hat{\partial}^{(k)}F_\nu (x^{(k)})$ can be used to successfully approximate $F_\nu$. To generate the approximation matrix $B^{(k)}$, we specify an initial value $B^{(1)}\in \R^{n\times n}$ and update $B^{(k)}$ to satisfy the secant condition
\begin{align}
    \label{eq:secant}
    B^{(k+1)}\left(x^{(k+1)}-x^{(k)}\right) = \nabla f(x^{(k+1)})-\nabla f(x^{(k)}).
\end{align}
There are several update strategies that satisfy (\ref{eq:secant}), such as the Broyden method, and this paper uses the Broyden–Fletcher–Goldfarb–Shanno (BFGS) method in (\ref{eq:BFGS}). By updating with the BFGS formula, if $f$ is $\mu$-strongly convex and the initial value $B^{(1)}$ is a positive definite symmetric matrix, then $B^{(k)}>O$ holds. Thus, similar to Proposition \ref{prop:nonsingular}, it can be shown that all elements of $\hat{\partial}^{(k)}F_\nu (x^{(k)})$ are regular for all $k$. We present the procedure of the hybrid linear quasi-Newton method in Algorithm \ref{alg:quasi-newton}.

In \citep{xiao2018regularized}, a method utilizing the L-BFGS method is proposed, which approximates $\partial F_\nu$ using the L-BFGS method. The matrix updated with the L-BFGS method becomes symmetric, but $\partial F_\nu$ is generally not symmetric. Therefore, this approximation might not be accurate. Additionally, \citep{xiao2018regularized} does not provide a proof of convergence speed and requires the condition $\nu \leq 2L_f^{-1}$. 
We have proven that our proposed quasi-Newton method achieves superlinear convergence, which is a significant result.

\begin{algorithm}[h]
	\caption{(Hybrid Linear Quasi-Newton Method) input:$x^{(1)}$, output:$x^{\infty}$}
	\label{alg:quasi-newton}
        $\nu>0,B^{(1)}\in \R^{n\times n}$ is initialized, and steps \ref{enum:HQd}, \ref{enum:HQx}, and \ref{enum:HQB} are repeated for $k = 1,2,\ldots$ until convergence.
	\begin{enumerate}
		\item (Obtaining $d$)
		\label{enum:HQd}
				
$I-V^{(k)}(I-\nu B^{(k)})\in \hat{\partial}^{(k)} F_\nu (x^{(k)})$ for $V^{(k)} \in \partial_B \prox_{\nu g}(x^{(k)}-\nu \nabla f(x^{(k)}))$ is chosen, and then the search direction $d^{(k)}\in \R^n$ is selected as
		\begin{align}
			\label{eq:quasiup}
			d^{(k)} := -\left(I - V^{(k)}\left(I-\nu B^{(k)}\right)\right)^{-1}F_\nu (x^{(k)}).
		\end{align}
		\item(Updating $x$)
        \label{enum:HQx}
		\begin{align}
			x^{(k+1)} := x^{(k)} + d^{(k)}.
		\end{align}
		\item (Updating $B$)
		\label{enum:HQB}
		Let $s^{(k)} = x^{(k+1)}-x^{(k)},y^{(k)} = \nabla f(x^{(k+1)}) - \nabla f(x^{(k)})$. $B$ is updated as
		\begin{align}
			\label{eq:BFGS}
			B^{(k+1)} := B^{(k)} - \frac{B^{(k)} s^{(k)} (s^{(k)})^T B^{(k)}}{(s^{(k)})^T B^{(k)} s^{(k)}}+\frac{y^{(k)} (y^{(k)})^{T}}{(y^{(k)})^T s^{(k)}}.
		\end{align}
        \item When $x^{(k)}$ converges,
        \begin{align*}
            x^{\infty} = x^{(k+1)}.
        \end{align*}
	\end{enumerate}
\end{algorithm}

\subsection{Efficiency}
In both Algorithms \ref{alg:L-newton} and \ref{alg:quasi-newton}, solving (\ref{eq:newtonup}) and (\ref{eq:quasiup}) requires computing the inverse of an $n\times n$ matrix, which requires $O(n^3)$ computations. If $n$ is large, computing the inverse matrix requires an enormous amount of time, and it is inefficient to perform the computation in its current form.
Thus, we consider finding the search direction $d^{(k)}$ that satisfies the linear equation
\begin{align}
    \left(I - V^{(k)}\left(I-\nu B^{(k)}\right)\right) d^{(k)} = - F_\nu(x^{(k)}).
\end{align}
In the standard Newton method, the Newton-CG method, which combines the conjugate gradient (CG) method to efficiently solve linear equations, is widely used.
However, the CG method is only applicable to linear equations when the matrix is symmetric. In this case,
\[
\left(I - V^{(k)}\left(I-\nu B^{(k)}\right)\right)
\]
is generally not symmetric, so the conjugate gradient method cannot be used.
Thus, using the generalized conjugate residual (GCR) method, we find $d^{(k)}$ such that
\begin{align}
    \label{eq:gcr-t}
    \left\|F_\nu(x^{(k)})+\left(I - V^{(k)}\left(I-\nu B^{(k)} \right)\right)d^{(k)}\right\|_2\leq \epsilon^{(k)} \|F_\nu(x^{(k)})\|_2,
\end{align}
where $\epsilon^{(k)} > 0$ such as $\epsilon^{(k)} = \frac{1}{k+1}$ is the tolerance defined by the user.
The search direction $d^{(k)}$ is obtained by executing the GCR  method until the approximation error is acceptable. Specifically, we can express Step \ref{enum:HQd} of Algorithm \ref{alg:quasi-newton} as follows.
\begin{enumerate}
    \item
    $I-V^{(k)}(I-\nu B^{(k)})\in \hat{\partial}^{(k)} F_\nu (x^{(k)})$ for $V^{(k)} \in \partial_B \prox_{\nu g}(x^{(k)}-\nu \nabla f(x^{(k)}))$ is chosen, then the search direction $d^{(k)}\in \R^n$ that satisfies (\ref{eq:gcr-t}) is found using the GCR method.
\end{enumerate}

We specify the GCR method for finding $x$ such that $Ax = b$ for $A\in \R^{n \times n},b \in \R^n$ in Algorithm \ref{alg:GCR}.

If the GCR method can find $d^{(k)}$ that satisfies (\ref{eq:gcr-t}) in a finite number of steps, it can be updated only by multiplying a matrix and vector, and the GCR method can update using only $O(n^2)$ computations. We can obtain $d^{(k)}$ in fewer updates, which is more efficient than the $O(n^3)$ computations required to calculate the inverse matrix.
\begin{algorithm}[h]
	\caption{(Generalized Conjugate Residual Method) input: $x^{(1)}$, output: $x^{\infty}$}
	\label{alg:GCR}
		Let $r^{(1)} := b - Ax^{(1)}$ and $p^{(1)} := r^{(1)}$, and steps 1, 2, 3, and 4 are repeated for $k = 1,2,\ldots$ until convergence.
	\begin{enumerate}
				
\item $\alpha^{(k)} := \frac{\langle A p^{(k)}, r^{(k)} \rangle}{\langle Ap^{(k)},Ap^{(k)} \rangle}$
		\item(Updating $x,r$)
		\begin{align}
			x^{(k+1)} &:= x^{(k)} + \alpha^{(k)} p^{(k)} \\
			r^{(k+1)}  &:= r^{(k)} - \alpha^{(k)} A p^{(k)}
		\end{align}
		\item $\beta_{i,k} := - \frac{\langle A p^{(i)}, A r^{(k+1)} \rangle }{\langle Ap^{(i)},Ap^{(i)} \rangle} \;\;\;\; (i = 1,\ldots , k)$
				
\item $p^{(k+1)} := r^{(k+1)} + \sum_{i = 1}^k \beta_{i,k}A p^{(i)}$
\item When $x^{(k)}$ converges,
        \begin{align*}
            x^{\infty} = x^{(k+1)}.
        \end{align*}
	\end{enumerate}
\end{algorithm}\

\subsection{Convergence}
We consider the convergence properties of Algorithm \ref{alg:quasi-newton}. If $\partial_B \prox_{\nu g}$ is an LNA of $\prox_{\nu g}$, then we can show local linear convergence as follows. In this paper, we prove the following theorem in the same way as in \citep{sun1997newton} (see Appendix \ref{apendix:quasi-linear} for the proof).
\begin{theorem}
    \label{thm:quasi-linear}
    Suppose that $\partial_B \prox_{\nu g}$ is an LNA of $\prox_{\nu g}$ at $x^*-\nu \nabla f(x^*)$. There exist $\epsilon>0$ and $\Delta>0$ such that if $\|x^{(1)} - x^* \|_2 <\epsilon$ and $\|B^{(k)} - \nabla^2 f(x^{(k)})\|< \Delta$ for any $k = 1,2,\ldots$, then the sequence generated by Algorithm \ref{alg:quasi-newton} converges locally linearly to $x^*$.
\end{theorem}

From Theorem \ref{thm:quasi-linear}, if $B^{(k)}$ sufficiently approximates $\nabla^2 f(x^{(k)})$ and the initial value $x^{(1)}$ is sufficiently close to the optimal value $x^*$, then Algorithm \ref{alg:quasi-newton} exhibits first-order convergence.

Moreover, we can show that Algorithm \ref{alg:quasi-newton} leads to faster than linear convergence, in addition to the condition with respect to $\nabla^2 f$.
In this paper, we prove the following theorem.

\begin{theorem}
    \label{thm:quasi-super}
    Suppose that $\partial_B \prox_{\nu g}$ is an LNA of $\prox_{\nu g}$ at $x^* - \nu \nabla f(x^*)$, $\nabla^2 f$ is Lipschitz continuous, and the sequence $\{x^{(k)}\}$ generated by Algorithm \ref{alg:quasi-newton} satisfies $x^{(k)} \neq x^*$ for any $k$ and $\lim_{k \rightarrow \infty} x^{(k)} = x^*$.
    Then, $\{x^{(k)}\}$ superlinearly converges to $x^*$ if and only if $B^{(k)}$ satisfies
    \begin{align}
        \label{eq:superlinear}
		\lim_{k\rightarrow \infty}\frac{\|(B^{(k)} - \nabla^2 f(x^*))(x^{(k+1)}-x^{(k)})\|_2}{\|x^{(k+1)}-x^{(k)}\|_2} = 0
    \end{align}
    for any $k$.
\end{theorem}

Condition (\ref{eq:superlinear}) is similar to the condition for superlinear convergence of the standard quasi-Newton method. The BFGS formula (\ref{eq:BFGS}) used in this study satisfies (\ref{eq:superlinear}), therefore $\{x^{(k)} \}$ generated by Algorithm \ref{alg:quasi-newton} converges superlinearly.

\section{Numerical Experiments}
\label{sec:experiment}
To evaluate the performances of the linear Newton and hybrid linear quasi-Newton methods, we apply them to the problem of sparse estimation and compare them with the proximal gradient and proximal Newton methods through numerical experiments.
All programs are implemented using Rcpp.

\subsection{Group Logistic Regression}
Let $(y_i, x_i)\in \{-1,1\} \times \R^n, i = 1,\ldots, m$ where $m$ is the number of observations.
We formulate the optimization of group logistic regression as follows:
\begin{align}
    \min_{\beta_0\in \R,\beta \in \R^n} \frac{1}{m}\sum_{i = 1}^m \log(1 + \exp\{-y_i(\beta_0 + x_i^T \beta)\}) + \lambda \sum_{j = 1}^J \|\beta_{I_j} \|_2 ,
\end{align}
where $\lambda$ is a regularization parameter and $I_j, j = 1,\ldots,J$ are the index sets that belong to the $j$-th group such that $I_j \cap I_k = \emptyset \;(j\neq k)$.

In first experiment, we have only one group:
\begin{align}
    \label{eq:logistic}
    \min_{\beta_0\in \R,\beta \in \R^n} \frac{1}{m}\sum_{i = 1}^m \log(1 + \exp\{-y_i(\beta_0 + x_i^T \beta)\}) + \lambda\|\beta \|_2 .
\end{align}
We choose
\[
V = 
	\begin{cases}
		O, & \|x\|_2\leq \nu \\
		\frac{\nu}{\|x\|_2}(\frac{xx^T}{\|x\|_2^2}-I) + I, & \|x\|_2>\nu 
	\end{cases}
\]
as $V\in \partial_B \prox_{\nu g}(x)$. Figures \ref{fig:lambda1} and \ref{fig:lambdamax} compare the proximal gradient (PG) and proximal Newton (PN) methods with the proposed linear Newton (LN), hybrid linear quasi-Newton (HLQN), and hybrid linear quasi-Newton + GCR (HLQN-GCR) methods.
In both figures, the vertical and horizontal axes correspond to the value of $F_1(x^{(k)})$ and the computation time, respectively, which indicate the convergence of the algorithms. 

We generated random data with features $n = 2000$ and sample size $m = 4000$ and set the initial values $\beta_0^{(1)} = 0, \beta^{(1)} = 0, B^{(1)} = \nabla^2 f((\beta_0^{(1)}, \beta^{(1)}))$. 
Furthermore, $B^{(k)}$ was updated using the BFGS formula (\ref{eq:BFGS}) and $\epsilon^{(k)} = 0.001$.
Figure \ref{fig:lambda1} shows the graph when $\lambda$ is small ($\lambda = 1)$, i.e., when all features are active.
The LN, HQLN, and HQLN-GCR methods converge rapidly, and the LN method achieves a highly accurate solution.
The proximal Newton method is also a quadratically convergent algorithm, but in this case, the convergence is slow because the subproblem is solved by the proximal gradient method, and the convergence of the proximal gradient method is slow and stops halfway.
Moreover, the GCR method applied to the quasi-Newton method is faster than the HLQN method.

Figure \ref{fig:lambdamax} shows the graph when $\lambda$ is large, i.e., when all features are inactive.
Even in this case, the Newton and quasi-Newton methods converge rapidly and faster than in the case of small $\lambda$ in Figure \ref{fig:lambda1} due to the sparsity property.
The HLQN-GCR method converges the fastest.
For large $\lambda$, the convergence of the proximal Newton method is slower than in the small $\lambda$ case, and the convergence of the proximal gradient method stops in the middle of the convergence process.

\begin{figure}[h]
	\begin{tabular}{cc}
		\begin{minipage}{.5\textwidth}
			\begin{center}
				\resizebox{7cm}{!}{\input{figure1.tex}}
				\caption{Changes in $F_1(x^{(k)})$ due to the computation time. ($\lambda = 1$)}
				\label{fig:lambda1}
			\end{center}
		\end{minipage}
		\begin{minipage}{.5\textwidth}
			\begin{center}
				\resizebox{7cm}{!}{\input{figure2.tex}}
				\caption{Changes in $F_1(x^{(k)})$ due to the computation time. (large $\lambda$)}
				\label{fig:lambdamax}
			\end{center}
		\end{minipage}
	\end{tabular}
\end{figure}

We present the results of applying our proposed method to real-world data. The datasets used are cod-RNA\citep{uzilov2006detection} and ijcnn1\citep{prokhorov2001ijcnn}, obtained from the LIBSVM website\footnote{\url{https://www.csie.ntu.edu.tw/~cjlin/libsvmtools/datasets/}}, with sample sizes of $m = 59935$ and $49990$ respectively. To introduce a group structure to the data, second-order polynomial features were generated from the original features \citep{roth2008group, pavlidis2001gene}. The dimensions of the generated features are $n = 140$ and $1155$ respectively, with the number of groups being $J = 28$ and $231$. In this case, as the loss function in (\ref{eq:logistic}) is not strongly convex, we add ridge regularization, resulting in the optimization of:

\begin{align}
\label{eq:logistic_ridge}
\min_{\beta_0\in \R,\beta \in \R^n} \frac{1}{m}\sum_{i = 1}^m \log(1 + \exp\{-y_i(\beta_0 + x_i^T \beta)\}) + \frac{\epsilon}{2} \|\beta \|_2^2 +  \lambda \sum_{j = 1}^J \|\beta_{I_j} \|_2 
\end{align}

In this experiment, the ridge parameter was set to $\epsilon = 0.05$, and similar to the random data case, $B^{(1)} = \nabla^2 f((\beta_0^{(1)}, \beta^{(1)}))$ was used as the initial matrix, and updated using the BFGS formula.

Figures \ref{fig:codrnalambda008} and \ref{fig:codrnalambda028} show computation time graphs when applied to the cod-RNA dataset. The initial values $\beta_0^{(1)}$ and $\beta^{(1)}$ were set to be -0.55 and 0 respectively, and -0.55 is obtained as the optimal $\beta_0$ when $\beta = 0$. Figure \ref{fig:codrnalambda008} represents the case with $\lambda = 0.08$, where the number of active groups is 9, and the dimension of non-zero $\beta$ is 45. Similar to the case of random data, the proposed LN, HLQN, and HLQN-GCR methods demonstrate fast convergence. However, the proximal Newton method (PN) shows slower convergence.

Next, Figure \ref{fig:codrnalambda028} displays the results for the case with $\lambda = 0.28$, where the number of active groups is 1, and the dimension of non-zero $\beta$ is 5. The quasi-Newton methods, HLQN and HLQN-GCR, converge rapidly and exhibit efficient performance. While the proximal Newton method performs better compared to when $\lambda$ is small, it stops converging midway. The LN method is slow at first, but eventually, the LN method converges rapidly.

\begin{figure}[h]
	\begin{tabular}{cc}
		\begin{minipage}{.5\textwidth}
			\begin{center}
				\resizebox{7cm}{!}{\input{figure3.tex}}
				\caption{Change in computation time for the cod-RNA dataset. ($\lambda = 0.08$)}
				\label{fig:codrnalambda008}
			\end{center}
		\end{minipage}
		\begin{minipage}{.5\textwidth}
			\begin{center}
				\resizebox{7cm}{!}{\input{figure4.tex}}
				\caption{Change in computation time for the cod-RNA dataset. ($\lambda = 0.28$)}
				\label{fig:codrnalambda028}
			\end{center}
		\end{minipage}
	\end{tabular}
\end{figure}

Figures \ref{fig:ijcnnlambda008} and \ref{fig:ijcnnlambda012} show computation time graphs when applied to the ijcnn1 dataset. The initial values were determined similarly to the cod-RNA case, with $\beta_0^{(1)} = - 1.1 $ and $\beta^{(1)} = 0$.
Figure \ref{fig:ijcnnlambda008} represents the case with $\lambda = 0.08$, where the number of active groups is 49, and the dimension of non-zero $\beta$ is 245. On the other hand, Figure \ref{fig:ijcnnlambda012} displays the results for the case with $\lambda = 0.12$, where the number of active groups is 5, and the dimension of non-zero $\beta$ is 25.
In both cases, the proposed HLQN and HLQN-GCR methods converge rapidly, efficiently obtaining solutions. However, However, because of the high computational cost per update, LN takes a long time to get close to the optimal solution. As a result, it cannot fully leverage its fast convergence, leading to a long computational time. The proximal gradient method is initially fast, but it eventually converges slowly.

\begin{figure}[h]
	\begin{tabular}{cc}
		\begin{minipage}{.5\textwidth}
			\begin{center}
				\resizebox{7cm}{!}{\input{figure5.tex}}
				\caption{Change in computation time for the ijcnn1 dataset. ($\lambda = 0.08$)}
				\label{fig:ijcnnlambda008}
			\end{center}
		\end{minipage}
		\begin{minipage}{.5\textwidth}
			\begin{center}
				\resizebox{7cm}{!}{\input{figure6.tex}}
				\caption{Change in computation time for the ijcnn1 dataset. ($\lambda = 0.12$)}
				\label{fig:ijcnnlambda012}
			\end{center}
		\end{minipage}
	\end{tabular}
\end{figure}

\section{Conclusions}
\label{sec:conclusion}
In this paper, we considered the optimization problem of sparse estimation when the loss function is strongly convex and extended the Newton method with the proximal gradient step proposed in \citep{xiao2018regularized} to the general $\nu>0$ case. We proved the convergence of the method and provided theoretical guarantees, ensuring that the Newton method converges and is applicable even when the Lipschitz constant of $\nabla f$ is unknown. Furthermore, we proposed the HLQN method to approximate the second derivative $\nabla^2 f$ and theoretically proved that the HLQN method is always updatable and converges to the solution rapidly. When $\nabla^2 f$ is hard to compute, such as when the dimension $n$ of the data is large, the HLQN method can be executed faster than the Newton method, and thus, the solution can be obtained efficiently. Numerical experiments also showed that the proposed method is computationally efficient when applied to sparse estimation problems such as group logistic regression.

In this paper, we conducted numerical experiments on group regularization. Nonetheless, it is necessary to apply the method to other regularizations to further clarify the effectiveness of the method. We considered only local convergence, but it is necessary to consider global convergence using line search and other methods since it is not known whether the initial value is sufficiently close to the optimal value.

\appendix

\section{Proof of Proposition \ref{prop:nonsingular}}
\begin{lemma}
    \label{lem:ABeigen}
    Suppose $A\in \R^{n \times n}$ is a symmetric positive semidefinite matrix and $B\in \R^{n \times n}$ is a symmetric matrix. Then, any eigenvalue $\lambda$ of $AB$ satisfies $\min \{\|A \| \lambda_{\min} (B) ,0\} \leq \lambda  \leq \max\{ \|A\| \lambda_{\max} (B),0 \}$, where $\lambda_{\min} (B)$ and $\lambda_{\max} (B)$
    are the minimum and maximum eigenvalues, respectively, of $B$.
\end{lemma}
\begin{proof}
    Since the eigenvalues of $AB$ are equivalent to the eigenvalues of $BA$, we consider the eigenvalues of $BA$. We let $\lambda\in \R,x \in \R^n$ such that $BAx = \lambda x$ and $x\neq 0$. By multiplying $x^TA$ from the left, we obtain
    \[
        x^TABA x= \lambda x^T A x.
    \]
    If $x^T A x = 0$, then $\lambda = 0$ since $Ax = 0$. Next, we consider the $x^TAx > 0$ case. Since $A$ is a symmetric positive semidefinite matrix, there exists $A^{\frac{1}{2}}$, and we obtain
    \begin{align}
		\frac{x^TA B A x}{x^T A x} &= \lambda \nonumber \\
        \label{eq:A12}
		\frac{x^TA^{\frac{1}{2}} A^{\frac{1}{2}}B A^{\frac{1}{2}} A^{\frac{1}{2}}x}{x^T A^{\frac{1}{2}} A^{\frac{1}{2}} x} &= \lambda.
    \end{align}
    We can rewrite (\ref{eq:A12}) as
    \[
        \frac{y^TA^{\frac{1}{2}} B A^{\frac{1}{2}} y}{y^T y} = \lambda,
    \]
    where $y = A^{\frac{1}{2}} x\neq0$. Thus, since $\|A^{\frac{1}{2}} y \|_2 \leq \| A\|^{\frac{1}{2}} \|y\|_2$, we can obtain
    \[
	\|A\| \lambda_{min}(B) \leq \frac{y^TA^{\frac{1}{2}} B A^{\frac{1}{2}} y}{y^T y}   \leq \|A\| \lambda_{max}(B).
    \]
    Therefore, if $x^T A x>0$, then $\|A\| \lambda_{min}(B) \leq \lambda \leq \|A\| \lambda_{max}(B)$. Using also the result when $x^T A x = 0$, the Lemma holds.
\end{proof}

\begin{theorem}[\rm{\cite{stella2017forward}, Theorem 3.2}]
    \label{thm:vnorm}
    Suppose $g:\R^n \rightarrow (-\infty, \infty]$ is a closed convex function. Every $V\in \partial_B \prox_{\nu g}(x)$ is a symmetric positive semidefinite matrix that satisfies $\|V\| \leq 1$ for all $x\in \R^n$.
\end{theorem}

\begin{proof}[Proof of Proposition \ref{prop:nonsingular}]
    By assumption, since $\nabla^2 f(x) > \mu I$ for any $x\in \R^n$,
    \[
        I- \nu \nabla^2 f(x) < (1-\nu \mu ) I.
    \]
    Since every $V\in \partial_B \prox_{\nu g}(x)$ is a symmetric positive semidefinite matrix that satisfies $\|V\| \leq 1$ for all $x\in \R^n$ by Theorem \ref{thm:vnorm}, from Lemma \ref{lem:ABeigen},
    \[
        \lambda_{\max}\left(  V \left(I-\nu \nabla^2 f(x) \right)  \right)<1- \nu \mu.
    \]
    Thus, every eigenvalue of $I - V \left(I-\nu \nabla^2 f(x) \right)$ is a real number that is greater than or equal to $\nu \mu$, and $I - V \left(I-\nu \nabla^2 f(x) \right)$ is a nonsingular matrix.
\end{proof}

\section{Proof of Theorem \ref{thm:quasi-linear}}
\label{apendix:quasi-linear}
\begin{lemma}[\rm{\cite{ortega2000iterative} Lemma 2.3.2}]
\label{lem:norm_mat}
Let $A, C \in \R^{n \times n}$ and assume that $A$ is invertible, with $\|A^{-1} \| \leq \alpha $. If $\| A - C \| \leq \beta$ and $\beta \alpha < 1$, then $C$ is also invertible, and
\[
\| C^{-1}\| \leq \frac{\alpha}{(1-\alpha \beta)}
\]
\end{lemma}
\begin{proof}[Proof of Theorem \ref{thm:quasi-linear}]
    Let $V^{(k)} \in \partial_B \prox_{\nu g}(x^{(k)} - \nu \nabla f(x^{(k)}))$, $U^{(k)} := I - V^{(k)}\left(I-\nu \nabla^2 f(x^{(k)}) \right)\in \partial F_\nu (x^{(k)})$ and $W^{(k)} := I - V^{(k)}\left(I-\nu B^{(k)} \right) \in \hat{\partial}^{(k)}F_\nu (x^{(k)})$.
    From Proposition \ref{prop:nonsingular}, every eigenvalue of $U^{(k)}$ is a real number that is greater than or equal to $\nu \mu$, and
    \[
        \left\|\left(U^{(k)}\right)^{-1} \right\| \leq \frac{\sqrt{n}}{\nu \mu}.
    \]

    Let $\Delta = \frac{\mu}{5\sqrt{n}}$. Since $\partial F_\nu$ is the LNA of $F_\nu$ at $x^*$, there exists $\epsilon > 0$ such that
    \[
        \|F_\nu (x) - F_\nu(x^*) - U(x-x^*)\|_2 \leq \nu \Delta \|x-x^*\|_2
    \]
    for any $x\in B(x^*, \epsilon) := \{y\mid \|x-y\|_2 < \epsilon \},U\in \partial F_\nu (x)$.
    Since $W^{(k)} - U^{(k)} = \nu V^{(k)}\left(B^{(k)} - \nabla^2 f(x^{(k)}) \right)$ and $\|B^{(k)} - \nabla^2 f(x^{(k)})\|< \Delta$, we obtain $\|W^{(k)} - U^{(k)}\| \leq \nu \Delta $. By Lemma \ref{lem:norm_mat}, $W^{(k)}$ is invertible and
    \begin{align*}
        \left\|\left(W^{(k)} \right)^{-1} \right\| &\leq  \frac{\sqrt{n}/\nu \mu}{1-\sqrt{n}/\nu \mu \times \Delta}\\
        &= \frac{5}{4} \frac{\sqrt{n}}{\nu \mu}.
    \end{align*}
    Thus, if $\|x^{(k)} - x^* \|_2< \epsilon$, then we have
    \begin{align*}
        \|x^{(k+1)}-x^*\|_2 &= \|x^{(k)}-(W^{(k)})^{-1}F_\nu (x^{(k)})-x^*\|_2 \\
	&\leq \|(W^{(k)})^{-1}\| \|F_\nu (x^{(k)})-F_\nu(x^*)-W^{(k)}(x^{(k)} - x^*)\|_2 \\
	&\leq \|(W^{(k)})^{-1}\| \left[\|F_\nu (x^{(k)})-F_\nu (x^*)-U^{(k)}(x^{(k)} - x^*)\|_2 + \|W^{(k)}-U^{(k)}\|\|x^{(k)}-x^*\|_2\right]\\
	&\leq \frac{5}{4}\frac{\sqrt{n}}{\nu \mu } (\nu \Delta \|x^{(k)} - x^*\|_2)\\
	&<\frac{1}{2}\|x^{(k)} - x^*\|_2
    \end{align*}
    Therefore, there exists $\epsilon, \Delta$ such that the sequence generated by Algorithm \ref{alg:quasi-newton} locally linearly converges to $x^*$.
\end{proof}
\section{Proof of Theorem \ref{thm:quasi-super}}
\label{apendix:quasi-super}
\begin{proof}
    Let $V^{(k)} \in \partial_B \prox_{\nu g}(x^{(k)} - \nu \nabla f(x^{(k)}))$, $U^{(k)} := I - V^{(k)}\left(I-\nu \nabla^2 f(x^{(k)}) \right)\in \partial F_\nu (x^{(k)})$ and $W^{(k)} := I - V^{(k)}\left(I-\nu B^{(k)} \right) \in \hat{\partial}^{(k)}F_\nu (x^{(k)})$.
    We let $e^{(k)} = x^{(k)}-x^*,s^{(k)}=x^{(k+1)}-x^{(k)}$. We note that $s^{(k)} = e^{(k+1)} - e^{(k)}$ and $\{e^{(k)}\}$ and $\{s^{(k)}\}$ converge to $0$ since $\{x^{(k)}\}$ converges to $x^*$. From the update rule of Algorithm \ref{alg:quasi-newton}, we have
    \begin{align*}
        F_\nu (x^*) &= \left[ F_\nu(x^{(k)}) + W^{(k)}s^{(k)}\right] + \left[\left(U^{(k)} - W^{(k)}\right)s^{(k)}\right] - \left[ F_\nu (x^{(k)}) - F_\nu (x^*) - U^{(k)} e^{(k)}\right] - U^{(k)}e^{(k+1)} \\
        &=\left[\left(U^{(k)} - W^{(k)}\right)s^{(k)}\right] - \left[ F_\nu (x^{(k)}) - F_\nu (x^*) - U^{(k)} e^{(k)}\right] - U^{(k)}e^{(k+1)}.
    \end{align*}
    Since $F_\nu (x^*) = 0$ and $U^{(k)}$ is a nonsingular matrix,
    \begin{align*}
        U^{(k)}e^{(k+1)} &=\left[\left(U^{(k)} - W^{(k)}\right)s^{(k)}\right] - \left[ F_\nu (x^{(k)}) - F_\nu (x^*) - U^{(k)} e^{(k)}\right] \\
        e^{(k+1)} &= \left( U^{(k)} \right)^{-1}\left[\left(U^{(k)} - W^{(k)}\right)s^{(k)}\right] - \left( U^{(k)} \right)^{-1} \left[ F_\nu (x^{(k)}) - F_\nu (x^*) - U^{(k)} e^{(k)}\right].
    \end{align*}
    By assumption, since $\|(W^{(k)} - U^{(k)})s^{(k)}\|_2 = o(\|s^{(k)} \|_2)$,
    \[
        \|e^{(k+1)}\|_2 = o(\|s^{(k)}\|_2) + o(\|e^{(k)}\|_2) = o(\|e^{(k+1)}\|_2) + o(\|e^{(k)} \|_2).
    \]
    Thus, we obtain $\|e^{(k+1)}\|_2 = o(\|e^{(k)}\|_2)$, and since $e^{(k)} = x^{(k)} - x^*$, the sequence $\{x^{(k)}\}$ generated by Algorithm \ref{alg:quasi-newton} superlinearly converges to $x^*$.
\end{proof}
\end{document}